\documentclass[12pt,a4paper,reqno]{amsart}
\usepackage{amsmath,amssymb,amsthm,amscd,mathrsfs} 
\usepackage[T1]{fontenc}
\usepackage[latin1]{inputenc}
\usepackage[english]{babel}
\usepackage[mathscr]{eucal}
\usepackage{graphicx}
\usepackage{float}
\usepackage{tikz}
\usetikzlibrary{positioning}
\usetikzlibrary{shapes.geometric}
\usetikzlibrary{fit}
\usetikzlibrary{patterns}
\usepackage{caption}
\usepackage{tikz-cd}
\usepackage{cases}
\usepackage{hyperref}
\usepackage{chngcntr}

\usetikzlibrary{decorations.text}

\input{comment.sty}
\includecomment{FZ}
\includecomment{MM}
\counterwithin{equation}{section}

\renewcommand{\a}{\alpha}

\renewcommand{\th}{\theta}

\renewcommand{\l}{\lambda}

\newcommand{\n}{\nu}

\newcommand{\ph}{\phi}
\def\ph{\phi}
\def\md#1{\ \mbox{\rm(mod }{#1})}

\def\npp#1{N_{\ph}^+(#1)}
\def\ol{\overline}

\def\ph{\phi}

\newcommand{\Q}{{\mathbb Q}}
\newcommand{\Z}{{\mathbb Z}}
\newcommand{\N}{{\mathbb N}}
\newcommand{\F}{{\mathbb F}}

\def\md#1{\ \mbox{\rm(mod }{#1})}

\def\npp#1{N_{\ph}^+(#1)}

\newcommand{\aF}{\mathfrak a}

\newcommand{\pF}{\mathfrak p}

\newtheorem{theorem}{Theorem}[section]

\newtheorem{lemma}[theorem]{Lemma}
\newtheorem{corollary}[theorem]{Corollary}

\theoremstyle{definition}

\newtheorem{definitions}[theorem]{Definitions}

\theoremstyle{remark}
\newtheorem{example}[theorem]{Example}
\newtheorem{remark}[theorem]{Remark}
\newtheorem{examples}[theorem]{Examples}

\begin{document}
\title[]{On  common index divisors and monogenity of certain  number fields defined by trinomials of type  $x^{2^r}+ax^m+b$}

\textcolor[rgb]{1.00,0.00,0.00}{}
\author{  Hamid Ben Yakkou}\textcolor[rgb]{1.00,0.00,0.00}{}
\address{Faculty of Sciences Dhar El Mahraz, P.O. Box  1874 Atlas-Fes , Sidi mohamed ben Abdellah University,  Morocco}\email{beyakouhamid@gmail.com}
\maketitle
\vspace{0.3cm}
\begin{abstract}  
 Let $K = \Q(\th)$ be a  number with $\th$ a root  of an  irreducible trinomial  of  type $ F(x)= x^{2^r}+ax^m+b \in \Z[x]$.   In this paper,    based on the $p$-adic Newton polygon techniques applied on decomposition of primes in number fields and the classical index theorem of Ore \cite{Narprime, O},  we study the monogenity of $K$. More precisely, we prove that if $a$ and $1+b$ are both divisible by $32$, then $K$ cannot be monogenic. For $m=1$, we  provide   explicit conditions on $a$, $b$ and  $r$ for which $K$ is not monogenic. We also construct a family of irreducible trinomials  which are not monogenic, but their roots generate monogenic number fields. To illustrate our results, we give some computational examples.
\end{abstract}
\maketitle
\section{Introduction and statements of results}
Let $K$ be a number field  generated by a   root   $\th$ of a monic irreducible  polynomial $F(x) \in \Z[x]$ of degree $n$ and   $\Z_K$   its  ring of  integers. The field $K$ is called monogenic if there exists  a primitive element $\eta \in \Z_K$ such that $\Z_K= \Z[\eta]$, that is $(1, \eta, \ldots, \eta^{n-1})$ is an integral basis (called a power integral basis) in $K$. 
The problem of studying the   monogenity of number fields and constructing power integral bases is one of the most important problems in algebraic number theory. This problem is intensively studied by several  researchers in the last four decades (cf. \cite{ANH6IJNT,BGGy6,BerEVGymultiply,R,EG,G19, Gyoryrsurlespolynomes,Gyoryredecide,JP,LN,PP}).

In a series of his papers \cite{Gyoryrsurlespolynomes,Gyoryredecide,GyorySeminarFrensh},  Gy\H{o}ry  provided the first general algorithms for deciding whether $K$ is monogenic or not and for determining all power integral bases in $\Z_K$. He also studied  in  \cite{Gyoryrelative} and  \cite{Gyoryrdiscriminant} the  question of monogenity in  relative extensions. Further, he succeeded to  reduce  index form equations to system of unit equations (see  \cite{Gyory1998}).

For any element  $\eta$ of $\Z_K$, let  $ \mbox{ind}(\eta)$ denote the index of the subgroup  $\mathbb{Z}[\eta]$ in $\Z_K$. The index  of   $K$ is defined as follows:
\begin{eqnarray*}
	i(K) = \gcd \ \{ \mbox{ind}(\eta) \, | \,\eta \in \Z_K \, \mbox{and} \, K= \Q(\eta)  \}.
\end{eqnarray*}  A prime  $p$ dividing $ i(K)$ is called a prime common index divisor of $K$.  Remark that if $K$ is monogenic, then $i(K) = 1$. Thus, a field possessing a  prime common index divisor is not monogenic. 

 The existence of common index divisor was first established by Dedekind. He show that the cubic number field $K=\Q(\th),$ where $\th$ is a root of $x^3+x^2-2x+8$ cannot be monogenic, since the prime    $2$ splits completely in $\Z_K$  (see e.g. \cite[page 64]{Na}).   There is an extensive literature on  indices of number fields, see e.g  \cite{R} by Dedekind,  \cite{Engstrom} by  Engstrom  and  \cite{Nartindex} by Nart.  

  In  \cite{LN}, Llorente and  Nart studied the index  of cubioc number fields defined by $x^3+ax+b$. In \cite{DS}, Davis and Spearman calculated the index of the quartic number field defined by $x^4+ax+b$.  In \cite{GPP41993}, Ga\'{a}l,  Peth\H{o} and Pohst gave efficient algorithm  for quartic number fields.  In \cite{PethoZigler}, Peth\H{o} and Ziegler gave an efficient criterion to decide whether the maximal order of a biquadratic field has a unit power integral basis or not.  In \cite{PP}, Peth\H{o} and Pohst studied  indices in  multiquadratic number fields.
  
   Combining  a refined version  \cite{Gyory1998}   of  the general  approach of \cite{Gyoryrsurlespolynomes, Gyoryredecide}    with an   efficient reduction and enumeration algorithms,   Ga\'al and  Gy\H{o}ry \cite{GGy5},  Bilu, Ga\'al and Gy\H{o}ry \cite{BGGy6} described  algorithms to solve index form equations in quintic resp sextic fields.  In \cite{BerEVGymultiply}, 	 Bérczes,  Evertse and Gy\H{o}ry  studied multiply monogenic orders.  The books \cite{ EG2015,EG} by Evertse and Gy\H{o}ry  gave detailed surveys on the discriminant form and index form theory and its applications, including related Diophantine equations and monogenity of number fields.
   
   Nakahara's research team based their method on the existence of    relative power integral bases of some special sub-fields, they studied the monogenity of several number fields: for example,  they  studied  in \cite{ANH6IJAC, ANH6IJNT}  the monogenity the pure sextic number field $\Q(\sqrt[6]{m})$.   
   
   Ben Yakkou \cite{BATCA1} and Ben Yakkou with Boudine \cite{BBAMH8}  studied  indices and  monogenity of  number fields defined by $x^8+ax+b$. In \cite{JonesASM2021}, Jones gave infinte families of non-monogenic trinomails. Also  in \cite{JP}, Jones and Phillips identify   classes  of monogenic trinomilas.

The purpose of this paper is to  study  the monogenity  of the  number  field  $K= \Q(\th)$ with   $\th $ a root  of an irreducible   trinomial of type  $F(x)= x^{2^r} +ax^m+b\in \Z[x]$. Recall that  in \cite{BFT1} Ben Yakkou and El Fadil studied the non-monogenity of number fields defined by $x^n+ax+b$. More precisely, they gave sufficient conditions on $a, b$ and $m$ for which $K$  admits an odd prime  common index divisor. Theses results were generalized in \cite{BRM} by Ben Yakkou for number fields defined by $x^n+ax^m+b$.    Also, in \cite{JKS}, Jakhar,   Khanduja and  Sangwan  studied the problem of the integral closedness of $\Z[\th]$: they   gave necessary and sufficient conditions for a prime $p$ to be a  divisor of the index $\mbox{ind}(\theta)$.  However, by Definition  of $i(K)$, the divisibility of $\mbox{ind}(\theta)$ by   $p$ is not sufficient to decide if  $p$ is a common index divisor of $K$ or not. Therefore, their results does not characterize  the prime divisors of indices of these number fields.  Therefore, the  results obtained in \cite{BRM, BFT1, JKS}  cannot give a complete  answer about the monogenity of number fields defined  by $x^{2^r}+ax^m+b$.   For this reason, we have chosen to study this special case separately.\\
Throughout this paper,  for any prime  $p$ and any $t \in \Z$, $\n_p(t)$ stands for the $p$-adic valuation of $t$  and  $t_p:=\frac{t}{p^{\nu_p(t)}}$.   We recall also  that  the discriminant of the trinomial $F(x) = x^n +ax^m +b $ is 
 \begin{eqnarray}\label{disctrinom}
 \Delta(F)= (-1)^{\frac{n(n-1)}{2}} b^{m-1} (n^{n_1} b^{n_1-m_1} - (-1)^{m_1} m^{m_1}(m-n)^{n_1 - m_1}a^{n_1})^{d_0},
 \end{eqnarray}
 where $d_0 = \gcd(n , m)$, $n_1 = \frac{n}{d_0} $ and $m_1= \frac{m}{d_0}$.  Recall also that the polynomial $F(x)$ is called monogenic if $\Z_K = \Z[\th]$.  It is important to note that the monogenity of  the polynomial $F(x)$ implies the monogenity of the field $K$. But the converse is not true.  Let us start with the following theorem, which is in a more general case. It gives infinite parametric families of irreducible non-monogenic trinomials with non-squarefree discriminant, but their roots generate monogenic number fields.
\begin{theorem}\label{mono}\
\\Let $F(x)=x^n+ax^m+b \in \Z[x]$ be a monic polynomial with discriminant $\Delta$. Suppose that there exist a prime  $p$ dividing both $a$ and $b$ such that $\n_p(b)\ge 2$, $\gcd(n, \n_p(b))=1$, $ n \n_p(a) > (n-m) \n_p(b),$ and $\Delta_p$ is square free. Then $F(x)$ is irreducible over $\Q$. Let $K=\Q(\th)$ be a number field with $\th$ a  root of $F(x)$. Then $F(x)$ is not monogenic ($\Z_K \neq \Z[\th]$), but $K$ is monogenic. Moreover, in this case,  $\Z_K=\Z[\a],$  with  $\a=\frac{\th^s}{p^t},$ where   $(s,t)\in \N^2$  is  the unique positive  solution of the Diophantine equation $\n_p(b)s-nt=1$ with $0\le s <n$.
\end{theorem}
\begin{remark}
Theorem \ref{mono} implies \cite[Theorem 2.1]{BFT1}, where the special case   $n=p^r$ and $m=1$ is previously considered. 
\end{remark}
\begin{corollary}\label{cor}\
\\Let $K=\Q(\th)$ be  a number field with $\th$ a root of a monic irreducible polynomial $F(x)=x^{2^r}+ax^m+b\in \Z[x]$.  Assume that  there exist a prime  $p$ such that $\n_p(a) \ge \n_p(b) \ge 3$, $\n_p(b)$ is odd, and $\Delta_p$ is square free. Then $F(x)$ is irreducible over $\Q$ and not monogenic. But,  $K$ is monogenic.
\end{corollary}
\begin{example}\
\\The polynomial $F(x)=x^8+8x+8 \in \Z[x]$ has discriminant $\Delta = 2^{24}  \times 1 273609$. Note that $1273609$ is a prime. Then $F(x)$ satisfies the conditions of Corollary \ref{cor} for $p=2$. Hence, it is irreducible over $\Q$. Let $K=\Q(\th)$ with $\th$ is a root of $F(x)$. Then $\Z_K \neq \Z[\th]$. But, $K$ is monogenic and $\a = \frac{\th^3}{2}$ generates a power integral basis of $\Z_K$.
\end{example}

The following result gives infinite parametric  families number fields defined by $x^{2^r}+ax^m+b$ their indices are divisible by $2$ which garantie   the  non monogenity of these fields.  
\begin{theorem}\label{npib1}\
\\Let $F(x)=x^{2^r}+ax^m+b$ be a monic irreducible trinomial  and $K=\Q(\th)$ a number fields generated by $\th$, a  root of $F(x)$. If $r\ge 3$ and $a$ and $b+1$ are both divisible by $32$,
	then  $2$ is a prime common index divisor of $K$.  In particular,  $K$ is not monogenic.
\end{theorem}

Now, we focus on the case $m=1$. The following Theorem gives  sufficient  conditions which guarantee the non monogenity of infinite parametric families of number fields defined by trinomials of type $x^{2^r}+ax+b$.
\begin{theorem}\label{npib}\
\\Let $K=\Q(\th)$ be  a number field with $\th$ root of a monic irreducible trinomial $F(x)=x^{2^r}+ax+b\in \Z[x]$. If one of the following conditions holds 
\begin{enumerate}
\item $r \ge 3$, $a \equiv 4 \md 8$ and $b\equiv 3 \md 8$.
 \item  $r \ge 4$, $a \equiv 8 \md{16}$ and $b \equiv 7 \md{16}$.
 \item  $r\ge 3$ and  $(a,b) \in \{(0,31), (16,15)\} \md{32}$,

\end{enumerate}
	then  $2$ is a prime common index divisor of $K$.  In particular, if one of these conditions holds, then $K$ is not monogenic.
\end{theorem}
\begin{remark}
Theorem \ref{npib}(1)(3) implies respectively  \cite[Theorem 2.3(1)(2)]{BATCA1}  when the special case $r=3$ is previously studied.
\end{remark}
When $a=0$, the  following result is an immediate consequence of the above theorem.
\begin{corollary}
If $b \equiv 1 \md {32}$, then the pure number field $\Q(\sqrt[2^r]{b})$ is not monogenic for every natural integer $r \ge 3$. 
	
	\end{corollary}

\begin{remark}\
\\Note that \cite[Corollary 2.4]{BRM}  is about trinomials of  type $x^{2^r\cdot 3^k}+ax^m+b$. It gives sufficient conditions for the prime $3$ to be a common index divisor of $K$. Although
formally it  includes $ x^{2^r} + ax^m + b$, all statements of Corollary 2.4 of \cite{BRM} concerns
the cases $k \ge 1$. Hence do not overlap with  Theorems \ref{npib1} and  \ref{npib}. Here, we  gave sufficient conditions for which the prime $2$ is a common index divisor of $K$.
\end{remark}	

\section{Newton polygons and the index theorem of Ore}

To prove  our main results, we need some preliminaries that can be found in details in \cite{BATCA1,BRM}. For any $\eta \in \Z_K$, it is well known from  \cite[Proposition 2.13]{Na} that   
\begin{eqnarray}\label{indexdiscrininant}
D (\eta) = ( \Z_K : \Z [\eta])^2  \cdot  D_K,
\end{eqnarray}
where $D(\eta)$ is the discriminant of the minimal polynomial of $\eta$ and $D_K$ is the discriminant of $K$. Let $p$ be a prime. In 1878, Dedekind gave the explicit factorization of $p\Z_K$ when $p $ does not divide the index $ ( \Z_K : \Z [\eta])$ (see  \cite{R} and \cite[Theorem 4.33]{Na}). In 1928,   
\O. Ore \cite{O}  gave  a method for factoring $F(x)$ in $\Q_p(x)$, and so, for factoring  $p\Z_K$ when $F(x)$ is $p$-regular (see \cite{O}).   This  method is  based on Newton polygon techniques. His method was developed by   Gu\`{a}rdia, Montes and  Nart and  \cite{Narprime, Nar}, see also \cite{MN92} by   Montes and  Nart. So, let us  recall   some fundamental facts on  this  algorithm.

 Let  $\n_p$ be  the discrete valuation of $\Q_p(x)$  defined on $\Z_p[x]$ by $$\n_p\left(\sum_{i=0}^{m} a_i x^i\right) = \min \{ \n_p(a_i), \, 0 \le i \le m\}.$$ Let $\phi(x) \in \mathbb{Z}[x]$ be a monic polynomial whose reduction modulo $p$ is irreducible.  Upon the euclidean division by successive powers of $\ph(x),$ the  polynomial $F(x) \in \mathbb{Z}[x]$ admits a unique $\phi$-adic development $$ F(x )= a_0(x)+ a_1(x) \phi(x) + \cdots + a_n(x) {\phi(x)}^n,$$ with $ \deg \ ( a_i (x))  < \deg \ ( \phi(x))$. For every $0\le i \le n,$   let $ u_i = \n_p(a_i(x))$. The $\phi$-Newton polygon of $F(x)$ with respect to $\n_p$ (or to $p$, briefly) is  the lower convex hull of the points $ \{  ( i , u_i) \, | \,  0 \le i \le n \, , a_i(x) \neq 0  \}$ in the euclidean plane, which we denote   by $N_{\phi} (F)$. The polygon  $N_{\phi} (F)$ is the union of different adjacent sides $ S_1, S_2, \ldots , S_g$ with increasing slopes $ \lambda_1, \lambda_2, \ldots,\lambda_g$. We shall write $N_\phi(F) = S_1+S_2+\cdots+S_g$. The polygon determined by the sides of negative slopes of $N_{\phi}(F)$ is called the  $\phi$-principal Newton polygon of $F(x)$  with respect to $\n_p$ and will be denoted by $\npp{F}$. Note that the length of $\npp{F}$ is $ l(\npp{F}) = \nu_{\overline{\ph}}(\overline{F(x)})$;  the highest power of $\phi(x)$ dividing $F(x)$ modulo $p$.\\
Let $S$ be a side of $\npp{F}$. Then the length of $S$, denoted $l(S)$ is the length of its  projection to the horizontal axis and its height, denoted $h(S)$ is the length of its  projection to the vertical axis.  Let   $\lambda = - \frac{h(S)}{l(S)}= - \frac{h}{e} $  its slope, where $e$ and $h$ are two positive coprime integers.  The degree of $S$ is $d(S) = \gcd(h(S), l(S))=  \frac{l(S)}{e}$; it is equal to the the number of segments into which the integral lattice divides $S$. More precisely, if $ (s , u_s)$ is the initial point of $S$, then the points with integer coordinates  lying in $S$ are exactly $$  (s , u_s) ,\ (s+e , u_s - h) , \ldots, (s+de , u_s - dh).$$  The natural integer $e= \frac{l(S)}{d(S)}$ is called the ramification index of the side $S$ and denoted by $e(S)$.\\
Let $\mathbb{F}_{\phi}$ be the finite field   $ \mathbb{Z}[x]\textfractionsolidus(p,\phi (x)) \simeq \mathbb{F}_p[x]\textfractionsolidus (\overline{\ph(x)})$ (note that the ideal $(p,\phi(x))$ is maximal in the ring of  polynomials  $\Z[x]$).
For   any abscissa $ s \leq i \leq  s+de,$ we define  the following residue coefficient $ c_i \in  \mathbb{F}_{\phi}$:
$$c_{i}=
\left
\{\begin{array}{ll} 0,& \mbox{if }  (i , u_i ) \, \text{ lies  strictly  above }  \ \npp{F},\\
\dfrac{a_i(x)}{p^{u_i}}
\,\,
\md{(p,\phi(x))},&\mbox{if } \  (i , u_i ) \, \text{lies on } \npp{F}. 
\end{array}
\right.$$ Further, we attach to $S$ the  residual polynomial:  $$ R_{\l}(F)(y) = c_s + c_{s+e}y+ \cdots + c_{s+(d-1)e}y^{d- 1}+ c_{s+de}y^d \in \mathbb{F}_{\phi}[y].$$ Now, we recall some related definitions to Ore's program.
\begin{definitions}
	Let $F(x) \in \Z[x]$ be a monic irreducible polynomial.  Let $\overline{F(x)}=\prod_{i=1}^t\overline{\ph_i}(x)^{l_i}$ be  the factorization  of $\overline{F(x)}$ into a product  of powers of distinct monic irreducible polynomials in  $\mathbb{F}_p [x]$. For every $i=1,\dots,t$, let  $N_{\ph_i}^+(F)=S_{i1}+\dots+S_{ir_i}$, and for every {$j=1,\dots, r_i$},  let $R_{\l_{ij}}(F)(y)=\prod_{s=1}^{s_{ij}}\psi_{ijs}^{n_{ijs}}(y)$ be the factorization of $R_{\l_{ij}}(F)(y)$ in $\F_{\ph_i}[y]$. 
	\begin{enumerate}
		\item  For every $i=1,\dots,t$, the $\ph_i$-index of $F(x)$, denoted by $ind_{\ph_i}(F)$, is  deg$(\ph_i)$ multiplied by  the number of points with natural integer coordinates that lie below or on the polygon $N_{\ph_i}^{+}(F)$, strictly above the horizontal axis  and strictly beyond the vertical axis.
		\item The polynomial $F(x)$ is said to be $\phi_i$-regular with respect to $\n_p$  if for every $j=1,\dots, r_i$,  $R_{\l_{ij}}(F)(y)$ is separable; $n_{ijs}=1$.
		\item The polynomial $F(x)$ is said to be $p$-regular if it is $\phi_i$-regular for every $ 1 \leq i \leq t$.
	\end{enumerate}
\end{definitions} 
Now, we state Ore's theorem which will be often used in the proof of our results (see  \cite[Theorem 1.19]{Nar}, \cite{MN92} and \cite{O}):
\begin{theorem}\label{ore} (Ore's Theorem)\
	\\Let $K$ be a number field generated by  $\th$, a root of a monic irreducible polynomial $F(x) \in \Z[x]$.  Under the above notations, we have: 
	\begin{enumerate}
		\item
		$$ \nu_p((\Z_K:\Z[\th]))\ge \sum_{i=1}^t ind_{\ph_i}(F).$$ Moreover, the  equality holds if $F(x)$ is $p$-regular
		\item
		If  $F(x)$ is $p$-regular, then 
		$$p\Z_K=\prod_{i=1}^t\prod_{j=1}^{r_i}
		\prod_{s=1}^{s_{ij}}\pF^{e_{ij}}_{ijs},$$ where $e_{ij}$ is the ramification index
		of the side $S_{ij}$ and $f_{ijs}=\mbox{deg}(\ph_i)\times \mbox{deg}(\psi_{ijs})$ is the residue degree of $\mathfrak{p}_{ijs}$ over $p$.
	\end{enumerate}
\end{theorem}
\section{proofs of the main results}
	\begin{proof} [Proof of Theorem \ref{mono}]\
		\\Reducing modulo $p$, one has  $F(x) \equiv \ph^{n} \md p$, where $\ph(x) = x$, because $p$ divides both $a$ and $b$. Further, we have $\npp{F}=S$  has  only one  side of degree $1$ (because $\gcd(n,\n_p(b))=1$) with slope $\l= -\frac{\n_p(b)}{n}$. Its attached  residual polynomial  $R_{\l}(F)(y)$ is irreducible over $\F_{\ph} \simeq \F_p $ as it is of degree $1$. By Theorem of the  residual polynomial (\cite[Theorem 1.19]{Nar}), $F(x)$ is irreducible over $\Q_p$.  Hence, it is irreducible over $\Q$. On the other hand, by using Theorem \ref{ore}, we see that $$\n_p(\Z_K:\Z[\th]) \ge ind_{\ph}(F) = \dfrac{(n-1)(\n_p(b)-1)}{2} \ge 1.$$ Thus, $p$ divides $(\Z_K:\Z[\th]) $.  So, the polynomial $F(x)$ is not monogenic.\\ 	Let $L=\Q_p(\th)$ and $\omega$  the   unique  valuation  of $L$ extending $\nu_p$ (note that $\Q_p$ is a Henselian field).	Let  $(s,t)\in \Z^2$  be the unique solution of the Diophantine equation $\n_p(x)s-nt=1$ with $0\le s <n$ and $\a=\frac{\th^s}{p^t}$.  Note  that  $\a\in \Z_K$ if and only if $\omega(\a)\ge 0$. Since $\npp{F}=S$ has  a single side of slope $ \lambda = - \frac{\n_p(b)}{n}$, we conclude that $\omega(\th)=\frac{\n_p(b)}{n}$. So, $$\omega(\a)=\omega\big(\frac{\th^s}{p^t}\big)=s \omega(\th)-t = \frac{s \n_p(b)-tn}{n} = \frac{1}{n}.$$ Since $s$ and $n$ are coprime,  $K=\Q(\a)$. Let $H(x)$ be the minimal polynomial of $\a$ over $\Q$. By the formula relating roots and coefficients of a monic polynomial, we conclude that $$H(x)=x^n+\sum_{i=1}^{n}(-1)^is_ix^{n-i},$$ where $s_i=\displaystyle\sum_{k_1<\dots<k_i}\a_{k_1}\cdots\a_{k_i}$, and  $\a_{1},\dots, \a_{n}$ are the $\overline{\Q_p}$-conjugates of $\a$. Since there is a unique valuation extending $\nu_p$ to any algebraic extension of $\Q_p$, we conclude that $\omega(\a_i)=1/n$ for every $i=1,\dots,n$ (recall that the valuation $\omega$ is invariant under the $K$-embedding actions). Thus,   $\nu_p(s_{n}) = \omega(\a_{1}\cdots\a_{n})=1$ and  $\nu_p(s_{i})\ge i/n$ for every $i=1,\dots, n-1$. That means that $H(x)$ is a $p$-Eisenstein polynomial. Hence, $p$ does not divide the index $(\Z_K:\Z[\a])$. On the other hand,  every prime $q \neq p$ does not divide $(\Z_K:\Z[\th])$, because $\Delta_p$ is square free (see \cite[Proposition 2.13]{Na}). By definition of $\a$,   $p$ is the unique positive prime integer candidate to divide  $(\Z[\a] : \Z[\th])$. Consequently,   $\Z_K=\Z[\a]$.
	
	\end{proof}

	To prove Theorems \ref{npib1} and  \ref{npib}, we will use  the following lemma which gives a sufficient  condition for a prime  $p$ to be a prime common index divisor of a given  field $K$ (see \cite{R} and \cite[Theorems 4.33 and 4.34 ]{Na}).
\begin{lemma} \label{comindex}\
	\\Let  $p$ be a prime  and $K$  a number field. For every positive integer $f$, let $L_p(f)$ be the number of distinct prime ideals of $\Z_K$ lying above $p$ with residue degree $f$ and $N_p(f)$ be the number of monic irreducible polynomials of  $\F_p[x]$ of degree $f$. If  $L_p(f) > N_p(f)$ for some positive integer $f$, then $p$ is a common index divisor of $K$.
\end{lemma} 
\begin{remark}\
	\\Note that the condition $i(K)=1$ is not sufficient for the monogenity  of $K$. The index of the  pure cubic number field  $K=\Q(\sqrt[3]{175})$ equals $1$, but $K$ is not monogenic as its index form equation is $5x^3-7y^3 = \pm 1$ and has no integral solutions.
\end{remark}
\begin{proof}[Proof of Theorem \ref{npib1}]
	By hypothesis, $2$ divides $a$ and does not divide $b$. Thus, $\ol{F(x)} = \ol{\ph_1(x)}^{2^r}$ in $\F_2[x]$, where $\ph_1(x)=x-1$.  Write \begin{eqnarray}\label{dev}
	F(x)&=& (x-1+1)^{2^r}+ax^m+b \nonumber \\
	&=& \sum_{j=1}^{2^r} \dbinom{2^r}{j}\ph_1(x)^j+ax^m+1+b. 
	\end{eqnarray}
Since $a$ and $b+1$ are both divisible by $32$,  $\n_2(a)$ and $\n_2(1+b) \ge 5$. So,  $$\n_2\left(ax^m+1+b)\right)\ge 5.$$
Let	$ax^m+1+b=\sum_{j=1}^{m}b_j \ph_1(x)^j$  be the $\ph$-adic development of $ax^m+1+b$ where $b_j  \in \Z$. Note that $\n_2(b_j) \ge 5$ for all $j=1,\ldots,m$. It follows that $F(x)= \sum_{j=0}^{2^r}a_j \ph_1(x)^j,$ where $a_0=b_0$ and
$$a_{j}=
\left
\{\begin{array}{ll} b_j+\dbinom{2^r}{j}
,& \mbox{if } 1 \le j \le m,\\
\dbinom{2^r}{j},&\mbox{if } \  m<j \le 2^r. 
\end{array}
\right.$$
 Recall  that $\n_2\left(\dbinom{2^r}{j}\right)=r-\n_2(j))$ for every $j=1 , \ldots, 2^r-1$ (see  \cite[Lemma 3.4]{BRM}). It follows that $\mu_0=\n_2(a_0)\ge 5$ and
 $$
 \left
 \{\begin{array}{ll} \mu_j=\n_2\left(a_j\right) \ge \min\{5, r-\n_2(
 	j)\},& \mbox{if } 1 \le j \le m,\\
\mu_j=\n_2 \left( a_j\right)=r-\n_2(j),&\mbox{if } \  m<j \le 2^r. 
 \end{array}
 \right.$$
Since $r \ge 3$,  $N_{\ph_1}^{+}=S_{1,1}+\cdots+S_{1,t-2}+S_{1, t-1}+S_{1, t}$ has $t$ sides with $t\ge 4$ joining the points $\{(j, \mu_j),\,\, j=0,\ldots,2^r\}$ in the euclidean plane. The last three sides have degree $1$ each and ramification index  $2^{^{r-1}}$ each. More precisely, the part $S_{1, t-2}+S_{1, t-1}+S_{1, t}$ is the polygon joining the points $(2^{r-3},3), \,\, (2^{r-2},2), (2^{r-1},1) \,\, \mbox{and} \,\, (2^r,0)$ (see FIGURE 1).Thus the  residual polynomials $R_{1\,\, t-k}(F)(y)$  are irreducible in $\F_{\ph_2}[y]$ as they are of degree $1$ each for $k=0,1,2$.   Applying Theorem \ref{ore}, one has: $$2\Z_{K}=  \pF_{1, t-2, 1}^{2^{r-1}} \cdot \pF_{1, t-1, 1}^{2^{r-1}} \cdot \pF_{1, t, 1}^{2^{r-1}}\aF, $$ where $\pF_{1, t-2, 1},  \pF_{1, t-1, 1}$ and $\pF_{1, t, 1}$  are three prime ideals of $\Z_K$ of residue degree $1$ each, and    $\aF$ is a non-zero ideal of $\Z_{K}$ provided by the other segments of $N_{\ph_1}^{+}(F)$. So, the  monic irreducible factor $\ph_1(x)$ of $F(x)$ modulo $2$ provides at least three prime ideals of residue degree $1$ each lying above $2$ in $\Z_{K}$. So,   $L_2(1) \ge 3 > 2= N_2(1)$. By Lemma \ref{comindex}, $2$ divides $i(K)$. Hence, $K$ is not monogenic.

\begin{figure}[htbp] 
	\centering
	\begin{tikzpicture}[x=0.75cm,y=0.5cm]
	\draw[latex-latex] (0,7) -- (0,0) -- (17,0) ;

	\draw[thick] (0,0) -- (-0.5,0);
	\draw[thick] (0,0) -- (0,-0.5);
	
	\draw[thick,red] (2,-2pt) -- (2,2pt);
	\draw[thick,red] (4,-2pt) -- (4,2pt);
	\draw[thick,red] (8,-2pt) -- (8,2pt);
	\draw[thick,red] (-2pt,1) -- (2pt,1);
	\draw[thick,red] (-2pt,2) -- (2pt,2);
	\draw[thick,red] (-2pt,3) -- (2pt,3);
	\draw[thick,red] (-2pt,4) -- (2pt,4);	
	\draw[thick,red] (-2pt,5) -- (2pt,5);	
	\node at (2,0) [below ,blue]{\footnotesize  $2^{r-1}$};
	\node at (4,0) [below ,blue]{\footnotesize $2^{r-2}$};
	\node at (8,0) [below ,blue]{\footnotesize  $2^{r-1}$};
	\node at (16,0) [below ,blue]{\footnotesize  $2^r$};
	\node at (0,1) [left ,blue]{\footnotesize  $1$};
	\node at (0,2) [left ,blue]{\footnotesize  $2$};
	\node at (0,3) [left ,blue]{\footnotesize  $3$};
	\node at (0,4) [left ,blue]{\footnotesize  $4$};
	\node at (0,5) [left ,blue]{\footnotesize  $\mu_5$};
	\draw[thick, mark = *] plot coordinates{(2,3) (4,2) (8,1) (16,0)};
	\draw[thick, dashed] plot coordinates{(1,4) (2,3) };
	\draw[thick, only marks, mark=*] plot coordinates{  };
	\node at (3,2.5) [above  ,blue]{\footnotesize $S_{1\,\,t-2}$};	
	\node at (6,1.2) [above  ,blue]{\footnotesize $S_{1\,\,t-1}$};
	\node at (13,0.3) [above  ,blue]{\footnotesize $S_{1t}$};
	\end{tikzpicture}
	\caption{\small   $N_{\ph_1}^{+}(F)$ where $r \ge 3$ and $a \equiv 0 \md{32}$ and $1+b\equiv 0 \md{32}$ }
\end{figure}

\end{proof}

	\begin{proof} [Proof of Theorem \ref{npib}]\
	\\In all cases, we prove that $K$ is not monogenic by showing that that $2$ divides $i(K)$. Since in all cases, $2$ divides $a$ and  does not divide $b$, $\ol{F(x)}= \ol{(x-1)}^{2^r}$ in $\F_2[x]$. Write \begin{eqnarray}\label{dev}
		F(x)&=& (x-1+1)^{2^r}+a(x-1+1)+b \nonumber \\
		&=& (x-1)^{2^r}+ \sum_{j=2}^{2^r-1} \binom{2^r}{j}(x-1)^j+(2^r+a)(x-1)+1+a+b. 
	\end{eqnarray}
	Let $\ph_1(x) =x-1$, $\mu=\n_2(2^r+a)$ and $\n =  \n_2(1+a+b)$. It follows that by the above $\ph$-development (\ref{dev}) of $F(x)$ that the principal Newton polygon $\npp{F}$ with respect to $\n_2$ is the Newton polygon joining the points $\{(0, \n), (1, \mu)\} \cup \{(2^j, r-\n_2(j)), j=1, \ldots, r\}$. 
	\begin{enumerate}
	\item If $r \ge 3$, $a \equiv 4 \md 8$ and $b\equiv 3 \md 8$, then $\mu =2$ and $\n \ge 3$. It follows that $N_{\ph_1}^{+}{(F)}=S_{11}+S_{12}+S_{13}$ has three sides of degree $1$ each joining the points   $(0, \nu), (1,2), (2^{r-1},1) \,\, \mbox{and} \,\, (2^r,0)$ in the euclidean plane (see FIGURE 2 for $r=5$). Thus,  $R_{\l_{1k}}(F)(y)=1+y$ is irreducible in $\F_{\ph_1}[y] \simeq \F_2[y]$  as it is of degree $1$ for $i=1,2,3$. So, $F(x)$ is $2$-regular. Applying Theorem \ref{ore}, one gets:  $$2 \Z_K= \pF_{111} \cdot  \pF_{121}^{2^{r-1}-1} \cdot \pF_{131}^{2^{r-1}},$$ where $p_{1k1}$ is a  prime ideal of $\Z_K$ of residue degree $f(\pF_{1k1}/2)=1$ for $k=1,2,3$. So, there are three prime ideals of $\Z_K$ of residue degree $1$ each lying above $2$. Applying Lemma \ref{comindex} for $p=2$ and $f=1$, we see that $2$ divides $i(K)$. Consequently, $K$ cannot be monogenic.
		\begin{figure}[htbp] 
		\centering	
		\begin{tikzpicture}[x=0.4cm,y=0.5cm]
		\draw[latex-latex] (0,5) -- (0,0) -- (33,0) ;

		\draw[thick] (0,0) -- (-0.5,0);
		\draw[thick] (0,0) -- (0,-0.5);
		
		\draw[thick,red] (1,-2pt) -- (1,2pt);
		\draw[thick,red] (2,-2pt) -- (2,2pt);
		\draw[thick,red] (16,-2pt) -- (16,2pt);
		\draw[thick,red] (32,-2pt) -- (32,2pt);
		\draw[thick,red] (-2pt,1) -- (2pt,1);
		\draw[thick,red] (-2pt,2) -- (2pt,2);
		\draw[thick,red] (-2pt,3) -- (2pt,3);
		\draw[thick,red] (-2pt,4) -- (2pt,4);	
		\node at (1,0) [below ,blue]{\footnotesize  $1$};
		\node at (2,0) [below ,blue]{\footnotesize $2$};
		\node at (16,0) [below ,blue]{\footnotesize  $2^4$};
		\node at (32,0) [below ,blue]{\footnotesize  $2^5$};
		\node at (0,1) [left ,blue]{\footnotesize  $1$};
		\node at (0,2) [left ,blue]{\footnotesize  $2$};
		\node at (0,3) [left ,blue]{\footnotesize  $3$};
		\node at (0,4) [left ,blue]{\footnotesize  $\nu$};
		\draw[thick, mark = *] plot coordinates{(0,4) (1,2) (16,1) (32,0)};
		\node at (0.5,3) [above  ,blue]{\footnotesize $S_{1}$};
		\node at (10,1.2) [above   ,blue]{\footnotesize $S_{2}$};
		\node at (20,0.4) [above   ,blue]{\footnotesize $S_{3}$};
		\end{tikzpicture}
		\caption{ \small  $N_{\ph_1}^{+}{(F)}$ with respect to $\n_2$ when $r=5$, $a\equiv 4 \md 8$ and $b \equiv 3 \md 8$.}
	\end{figure}
\item If $r \ge 4$, $a \equiv 8 \md{16}$ and $b \equiv 7 \md{16}$, then $\mu =3 $ and $\n \ge 4$. It follows that $N_{\ph_1}^{+}{(F)}=S_{11}+S_{12}+S_{13}+S_{14}$ has $4$ sides of degree $1$ each joining the points $(0, \nu), (1,3), (2^{r-2},2), (2^{r-1}, 1)\,\, \mbox{and} \,\,(2^r,0)$ in the euclidean plane with respective slopes $\l_{11} \le -1, \l_{12}=\frac{-1}{2^{r-1}-1}, \l_{13} = \frac{-1}{2^{r-2}} \,\, \mbox{and} \,\, \l_{14} =\frac{-1}{2^{r-1}}$ (see FIGURE 3 for $r=6$).  Thus, $R_{\l_{1k}}(F)(y)$ are irreducible over $\F_{\ph_1} \simeq \F_2$ as they are of degree $1$. It follows that $F(x)$ is $2$-regular. By using Theorem \ref{ore}, we see that $$2\Z_K= \pF_{111} \cdot  \pF_{121}^{2^{r-2}-1} \cdot  \pF_{131}^{2^{r-2}} \cdot  \pF_{141}^{2^{r-1}},$$ where   $p_{1k1}$ are prime ideals of $\Z_K$  of residue degree $f(\pF_{1k1}/2)=1$ for $k=1,2,3,4$. Thus, for $p=2$, we have $4 = L_2(1) >  N_2(1)=2$. By using Lemma \ref{comindex}, $2$ divides $i(K)$. Hence, $K$ cannot be monogenic.

\begin{figure}[htbp] 
	\centering	
	\begin{tikzpicture}[x=0.2cm,y=0.5cm]
	\draw[latex-latex] (0,6) -- (0,0) -- (66,0) ;

	\draw[thick] (0,0) -- (-0.5,0);
	\draw[thick] (0,0) -- (0,-0.5);
	
	\draw[thick,red] (1,-2pt) -- (1,2pt);
	\draw[thick,red] (2,-2pt) -- (2,2pt);
	\draw[thick,red] (16,-2pt) -- (16,2pt);
	\draw[thick,red] (32,-2pt) -- (32,2pt);
	\draw[thick,red] (-2pt,1) -- (2pt,1);
	\draw[thick,red] (-2pt,2) -- (2pt,2);
	\draw[thick,red] (-2pt,3) -- (2pt,3);
	\draw[thick,red] (-2pt,4) -- (2pt,4);	
	\node at (1,0) [below ,blue]{\footnotesize  $1$};
	\node at (2,0) [below ,blue]{\footnotesize $2$};
	\node at (16,0) [below ,blue]{\footnotesize  $2^4$};
	\node at (32,0) [below ,blue]{\footnotesize  $2^5$};
		\node at (64,0) [below ,blue]{\footnotesize  $2^6$};
	\node at (0,1) [left ,blue]{\footnotesize  $1$};
	\node at (0,2) [left ,blue]{\footnotesize  $2$};
	\node at (0,3) [left ,blue]{\footnotesize  $3$};
	\node at (0,5) [left ,blue]{\footnotesize  $\nu$};
	\draw[thick, mark = *] plot coordinates{(0,5) (1,3) (16,2) (32,1) (64,0)};
	\node at (0.9,4) [above  ,blue]{\footnotesize $S_{1}$};
	\node at (10,2.2) [above   ,blue]{\footnotesize $S_{2}$};
	\node at (20,1.4) [above   ,blue]{\footnotesize $S_{3}$};
		\node at (40,0.4) [above   ,blue]{\footnotesize $S_{3}$};
	\end{tikzpicture}
	\caption{\small   $N_{\ph_1}^{+}{(F)}$ with respect to $\n_2$  where  $r=6$,  $a \equiv 8 \md{16}$ and $b \equiv 7 \md{16}$.}
\end{figure} 
\item The case $r=3$ is previously  studied in \cite[Theorem 2.3(2)]{BATCA1}; we have shown that $2$ divides $i(K)$. Also, by Theorem \ref{npib1},  we have  $r \ge 3, \, a \equiv 0 \md{32}\, \mbox{and}\,\, b\equiv 31 \md{32}$, then $K$ is not monogenic.   Assume now that   $r\ge 4, a \equiv 16 \md{32}\, \mbox{and}\,\, b \equiv 15 \md{32}$, then $\mu \ge 4$ and $\n \ge 5$. Thus,  $$N_{\ph_1}^{+}{(F)}= S_{11}+ \ldots+ S_{1 \,\, t-2}+S_{1 \,\, t-1}+S_{1t}$$ has $t $ sides with $t \ge 4$. The last three sides  have degree $1$ each. More precisely, for $k=0,1,2$, $S_{1, \,\, t-k}$ is the side joining the points   $(2^{r-k-1}, k+1)$ and $(2^{r-k}, k)$. Their respective slopes are:  $\l_{1\,\, t-2}= \dfrac{-1}{2^{r-3}}, \l_{1\,\, t-1}=\dfrac{-1}{2^{r-2}} \,\, \mbox{and} \,\, \l_{1t}= \dfrac{-1}{2^{r-1}}$. It follows that $R_{\l_{1,\, t-k}}(F)(y)$ is irreducible in $\F_{\ph_1}[y]$ as it is of degree $1$ for $k=0,1,2$. Hence, $F(x)$ is $2$-regular. By applying Theorem \ref{ore}, one has: $$2 \Z_K = \pF_{1, t-2,  1}^{2^{r-1}} \cdot  \pF_{1, t-1, 1}^{2^{r-2}} \cdot  \pF_{1,  t, 1}^{2^{r-3}}  \cdot \aF,$$ where $\aF$ is a proper ideal of $\Z_K$ provided by the other sides of $\npp{F}$, and $\pF_{1, t-k, 1}$ is a prime ideal of $\Z_K$ of residue degree  $f(\pF_{1,  t-k, 1}/2) = 1 $ for $k=1,2,3$. So, there are at least  three prime ideals of $\Z_K$ of residue degree $1$ each lying above $2$. By Lemma \ref{comindex}, $2$ divides $i(K)$. Consequently, $K$ cannot be monogenic.
	\end{enumerate}
	\end{proof}
\newpage 
	To illustrate our results, we give some numerical examples.
	\begin{examples}\
	\\Let $K=\Q(\th)$ be a number field generated by a root of a monic  irreducible polynomial $F(x)=x^{2^r}+ax^m+b$. 
	
	\begin{enumerate}
	\item  If $F(x)= x^8+12x+3$, then $F(x)$ is irreducible as it is a $3$-Eisenstein polynomial. By Theorem \ref{npib}(1), the field $K$ cannot be monogenic. Further,  since the polynomial $F(x)$ is $2$-regular, by  \cite{Nar}, a $2$-integral basis of $K$ is given by $(1, \th,  \th^2, \th^3, \frac{35+\th^4}{2}, \frac{21 + 15 \th + 3 \th^4 + \th^5}{2}, \frac{7+ 5 \th^2+ 3 \th^4 + 2 \th^5+ \th^6}{2}, \frac{13+ \th+ \th^2+ \th^3+ \th^4+ \th^5+ \th^6+ \th^7}{4} ).$\\
	Let $I_2(x_2, x_3, \ldots, x_8) = \pm 1$ be the  $2$-index form equation associated to the above $2$-integral basis of $K$ which is a Diophantine equation of degree $28$, where the coefficients are in $\Z_{(2)}$, the localization of $\Z$ at $p=2$ (see \cite{DS}).  This equation has no solution due to $K$ being non-monogenic. 
	\item   If $F(x)= x^{16}+24x^{15}+8$, then $\Delta(F) = 2^{90} \times 7  \times  43 \times  29778 017378 311761 855723 790106 195659$. By Corollary \ref{cor}, $F(x)$ is not monogenic, but $K$ is monogenic and $\a = \frac{\th^{11}}{4}$ generates a power integral basis of $\Z_K$.
	\item  The pure number field $\Q(\sqrt[64]{65})$ is not monogenic.
	
	\end{enumerate}
	
	\end{examples}
	

\end{document}